\documentclass[11pt]{amsart}
\usepackage{amsfonts}
\usepackage{hyperref}
\usepackage[utf8]{inputenc}
\usepackage[T1]{fontenc}
\usepackage[dvips]{graphicx}
\usepackage{amssymb}
\usepackage{amsmath}
\usepackage{dsfont}
\usepackage{color}
\usepackage{latexsym}
\usepackage{bbm}
\usepackage{color}
\usepackage{amsthm}
\usepackage{multicol}
\usepackage[top   = 2.75cm,
bottom = 2.50cm,
left   = 2.50cm,
right  = 2.00cm]{geometry}

\bibstyle{plain}
\theoremstyle{plain}
\newtheorem{theorem}{Theorem}
\newtheorem{lemma}[theorem]{Lemma}

\newtheorem{corollary}[theorem]{Corollary}

\newtheorem{definition}[theorem]{Definition}

\theoremstyle{remark}

\newtheorem{remark}[theorem]{Remark}

\begin{document}

\title[Fractal Random Variables Related to $G_2$-representation]{Fractal Random Variables\\
  Defined by Probability Distributions of Digits\\
  of Their $G_2$-representation Having Two Bases\\
  with Different Signs}
\author[M.\ V.\ Pratsiovytyi, O.\ M.\ Baranovskyi and I.\ M.\ Lysenko and
S.\ P.\ Ratushniak]{ M.~Pratsiovytyi \and O.~Baranovskyi \and I.~Lysenko \and S.~Ratushniak}

\newcommand{\eacr}{\newline\indent}
\address{M.V. Pratsiovytyi\eacr
Institute of Mathematics of NASU,
 Dragomanov Ukrainian State University, Kyiv,
Ukraine\acr ORCID 0000-0001-6130-9413}
\email{prats4444@gmail.com}

\address{O.M. Baranovskyi\eacr
Institute of Mathematics of the National Academy of Sciences of Ukraine,
Ukraine\acr ORCID 0009-0001-4495-1778}
\email{baranovskyi@imath.kiev.ua}

\address{I.M. Lysenko\eacr
Dragomanov Ukrainian State University, Kyiv,
Ukraine
\acr ORCID 0009-0000-5299-7787}
\email{i.m.lysenko@udu.edu.ua}

\address{S.P. Ratushniak\eacr
Institute of Mathematics of NASU,
 Dragomanov Ukrainian State University, Kyiv,
Ukraine\acr ORCID 0009-0005-2849-6233}
\email{ratush404@gmail.com}

\keywords{two-symbol $G_2$-representation of real numbers,
  $G_2$-cylinder,
  tail set,
  discrete probability distribution,
  singular probability distribution,
  point spectrum of probability distribution (set of atoms),
  continuous spectrum (minimal closed support),
  essential support of probability distribution,
  fractal Hausdorff--Besicovitch dimension.}

\date{\today}

\newcommand{\acr}{\newline\indent}

\begin{abstract}
In this paper, we study distributions of two random variables
  \begin{gather*}
    \tau = \tau_1 g_{1-\tau_1}
      + \sum_{k=2}^\infty \tau_k g_{1-\tau_k} \prod_{i=1}^{k-1} g_{\tau_i}
    \equiv \Delta^{G_2}_{\tau_1\tau_2...\tau_n...},\\
    \xi = \xi_1 g_{1-\xi_1}
      + \sum_{k=2}^\infty \xi_k g_{1-\xi_k} \prod_{i=1}^{k-1} g_{\xi_i}
    \equiv \Delta^{G_2}_{\xi_1\xi_2...\xi_n...},
  \end{gather*}
  where $g_0$ is a given number belonging to interval $[\frac{1}{2}; 1)$,
  $g_1\equiv g_0 - 1$,
  $(\tau_n)$ and $(\xi_n)$ are sequences of random variables
  taking the values $0$ and $1$,
  and $(\tau_n)$ is a sequence of random variables
  that form a Markov chain
  with positive initial probabilities $p_{0}$, $p_{1}$
  and matrix of transition probabilities
  \[
    \begin{pmatrix}
      p_{00} & p_{01} \\
      p_{10} & p_{11}
    \end{pmatrix},
  \]
  $(\xi_n)$ is a sequence of independent random variables
  taking the specified values with probabilities $p_{0n}$ and $p_{1n}$, respectively
  ($p_{0n}+p_{1n}=1$).
  We study structural, spectral, and fractal properties
  of distributions of $\tau$ and $\xi$.
  For random variable $\tau$, point spectrum (the set of atoms)
  and continuous spectrum (minimal closed support) of its distribution
  are studied exhaustively.
  In particular, fractal properties of the distribution are described in detail.
  We prove a theorem on the Lebesgue purity of distribution of random variable $\xi$
  (an analog of the Jessen--Wintner theorem),
  i.e., conditions for the distribution to belong to one of the types:
  pure discrete, pure absolutely continuous, and pure singular.
  A criterion of discreteness is found.
  For discrete case, a point spectrum is described
  and, for continuous case,
  a criterion of absolute continuity and singularity is established.
  If digits of $G_2$-representation of random variable $\xi$
  are identically distributed,
  we prove a mutual orthogonality of two different distributions
  and calculate fractal dimension of the essential support of distribution.
\end{abstract}

\maketitle

\section*{Introduction}

In mathematics, many various two-symbol systems of encoding
(representations) of real numbers~\cite{Pratsiovytyi:2022:DSK:en} are used.
Because of minimal alphabet such systems deserve a special attention
and are technically convenient.
There is a unique and relatively new system
with two bases having different signs among them.
This is the so-called $G_2$-representation~\cite{Pratsiovytyi:2020:GCT}.
Systems of encoding are used for various goals and roles,
particularly for development of metric and probabilistic theory of real numbers~\cite{Galambos:1976:RRN,Schweiger:1995:ETF,Pratsiovytyi:2022:CAS},
ergodic theory~\cite{Billingsley:1965:ETI,Schweiger:1995:ETF},
theory of locally complicated functions~\cite{Pratsiovytyi:2022:CAS}
(first of all, continuous nowhere monotonic and non-differentiable functions),
theory of singular probability measures~\cite{Pratsiovytyi:2005:PDR:en,Albeverio:2006:JWT,Pratsiovytyi:2022:CAS},
studying groups of transformations of space~\cite{Pratsiovytyi:2020:GCT,Pratsiovytyi:2024:UGC}, etc.
If spectrum of distribution of a random variable
(i.e., minimal closed support, or equivalently,
the set of points of growth for a probability distribution function)
or its essential support of density
is a structurally or mertically fractal set
(i.e., a scale-invariant set,
set of fractional Hausdorff--Besicovitch dimension,
anomalously fractal
or superfractal set),
then we say that the random variable and its distribution is fractal~\cite{Pratsiovytyi:1998:FPD:en}.

This paper is mostly devoted to $G_2$-representation of numbers
and its application to studying of random variables with fractal properties,
namely, random variables
defined by given distributions of digits of its $G_2$-representation.

Let us recall key notions and facts~\cite{Lysenko:2019:DSC:en,Pratsiovytyi:2020:GCT,Pratsiovytyi:2022:DSK:en}
related to $G_2$-representation of numbers that we need in the sequel.

Let $g_0$ be a given real number (parameter) from interval $[0,5;1)$,
$g_1=g_0-1$,
$A=\{0;1\}$ alphabet (set of digits),
$L=A\times A\times...$ set of all infinite sequences of zeros and ones,
$\delta_0=0$, $\delta_1=g_0$, i.e., $\delta_i=ig_{1-i}$ for all $i \in A$.

\begin{theorem}
  For any number $x\in [0;g_0]$,
  there exists a sequence $(\alpha_n)\in L$ such that
  \begin{equation}
    \label{r3}
    x = \delta_{\alpha_1}
      + \sum_{k=2}^\infty (\delta_{\alpha_k} \prod_{j=1}^{k-1} g_{\alpha_j})
    \equiv \Delta^{G_2}_{\alpha_1\alpha_2...\alpha_k...},
    \quad
    \delta_{\alpha_k}=\alpha_kg_{1-\alpha_k}.
  \end{equation}
\end{theorem}

\begin{corollary}
  Let $g_0=0,5$.
  Then, for any number $x\in [0;\frac{1}{2}]$,
  there exists a sequence $(\alpha_n)$ of zeros and ones such that
  \begin{equation}
    \label{r2}
    x = \frac{\alpha_1}{2}
      + \sum_{k=2}^\infty \frac{\alpha_k(-1)^{\sigma_k}}{2^k}
    = \frac{\alpha_1}{2}
      + \sum_{k=2}^\infty \frac{\alpha_k}{2^{k-\sigma_k}(-2)^{\sigma_k}}
    \equiv \Delta^G_{\alpha_1\alpha_2...\alpha_n...},
  \end{equation}
  where
  $\sigma_k=\alpha_1+\alpha_2+...+\alpha_{k-1}$.
\end{corollary}

Really, let us put $g_0=0,5$.
Then, from equalty~\eqref{r3}, we obtain equality~\eqref{r2}.

Symbolic notation $x=\Delta^{G_2}_{\alpha_1\alpha_2...\alpha_n...}$
is called the $G_2$-representation of a number $x$,
and $\alpha_n$ is called the $n$th digit of this $G_2$-representation.
Parentheses mean period in representation of a number.

Remark that relation between $G$-representation of numbers,
which is defined by equality~\eqref{r2} and is a partial case of $G_2$-representation,
and classic binary representation is expressed by equality
\[
  \Delta^{G}_{\alpha_1\alpha_2...\alpha_n...}
  = \Delta^2_{0a_1a_2...a_n...}
  = \frac{0}{2} + \sum\limits_{n=1}^\infty \frac{a_n}{2^{n+1}},
\]

$a_1 = \begin{cases}
  0 & \mbox{ if } \alpha_1=0, \\
  1 & \mbox{ if } \alpha_1=1;
\end{cases}$
\quad
$a_{n+1} = \begin{cases}
  \alpha_{n+1} & \mbox{ if } \alpha_1+...+\alpha_n \mbox{ is even}, \\
  1-\alpha_{n+1} & \mbox{ if } \alpha_1+...+\alpha_n \mbox{ is odd}.
\end{cases}$

Numbers of a countable set have two $G_2$-representations:
\[
  \Delta^{G_2}_{\alpha_1...\alpha_{m-1}01(0)}
  = \Delta^{G_2}_{\alpha_1...\alpha_{m-1}11(0)},
\]
they are called \emph{$G_2$-binary}.
Numbers that have a unique representation are called \emph{$G_2$-unary}.
Numbers $0$ and $g_0$ have a unique $G_2$-representation:
$0=\Delta^{G_2}_{(0)}$, $g_0=\Delta^{G_2}_{(1)}$.

\begin{definition}
  \emph{$G_2$-cylinder} of rank $m$ with base $c_1c_2...c_m$ is the set
  \[
    \Delta^{G_2}_{c_1c_2...c_m} = \{ x \colon
      x = \Delta^{G_2}_{\alpha_1\alpha_2...\alpha_n...}, \;
      \alpha_i=c_i, \; i = \overline{1,m}\}.
  \]
\end{definition}

$G_2$-cylinders have the following properties:

1) $\Delta^{G_2}_{c_1...c_m}=\Delta^{G_2}_{c_1...c_m0}\cup\Delta^{G_2}_{c_1...c_m1}$;

2) Cylinder
$\Delta^{G_2}_{c_1c_2...c_m}$
is an interval with endpoints
\[
  a = \delta_{c_1} + \sum_{k=2}^m (\delta_{c_k} \prod_{j=1}^{k-1} g_{c_j}),
  \quad
  b = a + g_0 \prod\limits_{j=1}^m g_{c_j};
\]

3) The length of a cylinder is equal to
$|\Delta^{G_2}_{c_1c_2...c_m}| = g_0 \prod\limits_{j=1}^{m}|g_{c_j}|$;

4) Cylinders of the same rank are not overlapping,
and the following basic metric ratio holds:
$|\Delta^{G_2}_{c_1...c_mi}| = |g_i| |\Delta^{G_2}_{c_1...c_m}|$.

5) For any sequence $(c_n)\in L$, the following equalty holds:
\[
  \bigcap_{m=1}^\infty \Delta^{G_2}_{c_1...c_m}
  = \Delta^{G_2}_{c_1...c_m...},
\]
and it is an argument to believe that a point is a cylinder of infinite rank.

By definition, put
$\nabla^{G_2}_{c_1...c_m} \equiv int \Delta^{G_2}_{c_1...c_m}
= \Delta^{G_2}_{c_1...c_m}
\setminus \{\min\Delta^{G_2}_{c_1...c_m};\max\Delta^{G_2}_{c_1...c_m}\}$.

The $G_2$-representations of numbers
$x_1=\Delta^{G_2}_{\alpha_1 \alpha_2 \cdots \alpha_n \cdots}$
and
$x_2= \Delta^{G_2}_{\beta_1 \beta_2 \cdots \beta_n \cdots}$
are said \emph{to have the same (common) tail}
if there exist positive integer numbers $k$ and $m$ such that
$\alpha_{k+j} = \beta_{m+j}$
for any $j \in N$.
It is written symbolically by $x_1\sim x_2$.

The set of all numbers from interval $[0;g_0]$ that have the same tail
is called a \emph{tail set},
i.e., a tail set is an element of the quotient set of all $G_2$-representations of numbers
by binary relation ``to have the same tail.''
Every tail set is countable and everywere dense in interval $[0;g_0]$.

\begin{remark}
  Uniqueness of $G_2$-representation of numbers
  in the family of all two-symbol representations (encodings) of real numbers
  appears in various aspects.
  In particular,

  1) Left-shift operator for digits of $G_2$-representation of numbers
  \[
    \omega(\Delta^{G_2}_{\alpha_1\alpha_2...\alpha_n...})
    \equiv \Delta^{G_2}_{\alpha_2...\alpha_n...} = \begin{cases}
      g_0^{-1}x & \mbox{if } 0<x\leq g_0^2, \\
      g_1^{-1}(x-g_0) & \mbox{if } g_0^2\leq x\leq g_0
    \end{cases}
  \]
  is a piecewise-linear function;

  2) two representations of $G_2$-binary point belong to the same tail set;

  3) inversor of digits of $G_2$-representation of numbers
  \[
    I(\Delta^{G_2}_{\alpha_1\alpha_2...\alpha_n...})
    = \Delta^{G_2}_{[1-\alpha_1][1-\alpha_2]...[1-\alpha_n]...}
  \]
  is a function discontinuous at every $G_2$-binary point.
\end{remark}

\begin{lemma}
  If
  $\zeta=\Delta^{G_2}_{\zeta_1\zeta_2...\zeta_n...}$
  a random variable that is uniformly distributed on $[0;g_0]$
  then digits $\zeta_n$ of its $G_2$-representation
  are independent identically distributed random variables with
  $P\{\zeta_n=i\}=|g_i|$.
\end{lemma}

\begin{proof}
  Since $\zeta$ has a uniform distribution on $[0;g_0]$,
  it does not have atoms,
  i.e., $P\{\zeta=x_0\}=0$ for any $x_0\in [0;g_0]$.
  In addition, its probability distribution function is $F(x)=\frac{x}{g_0}$ and
  \[
    P\{\zeta=0\}=P\{\zeta\in\Delta_{0}^{G_2}\}=g_0^{-1}|\Delta^{G_2}_{0}|=g_0,
  \]
  \[
    P\{\zeta=1\}=P\{\zeta\in\Delta_{1}^{G_2}\}=g_0^{-1}|\Delta^{G_2}_{1}|=-g_1.
  \]
  Moreover,
  \[
    P\{\zeta\in \Delta^{G_2}_{c_1c_2...c_n}\}
    =g_0^{-1}|\Delta^{G_2}_{c_1c_2...c_n}|
    =g_0^{-1}\prod_{i=1}^{n}|g_{c_i}|.
  \]
  Then
  \[
    P\{\zeta_{n+1}=i/\xi\in\Delta^{G_2}_{c_1c_2...c_n}\}
    = \frac{P\{\zeta\in\Delta^{G_2}_{c_1...c_ni}\}}
           {P(\zeta\in\Delta^{G_2}_{c_1...c_n})}
    = \frac{g_0^{-1}|g_i|\prod\limits_{j=1}^{n}|g_{c_j}|}
           {g_0^{-1}\prod\limits_{j=1}^{n}|g_{c_j}|}=|g_i|.
  \]
  Since the last probability does not depend on choice of digits $(c_1,c_2,...,c_n)$,
  we see that $\xi_{n+1}$ does not depend on $\zeta_1$, $\zeta_2$, ..., $\zeta_n$.
  Thus, $(\zeta_n)$ is a sequence of independent random variables.
  Moreover,
  $P\{\zeta_n=0\}=g_0$, $P\{\zeta_n=1\}=-g_0$,
  and this completes the proof.
\end{proof}

\section{The First Object of Study}

Let $(\tau_n)$ be a sequence of random variables
taking the values $0$ and $1$
and forming a Markov chain
with positive initial probabilities $p_{0}$ and $p_{1}$
and matrix of transition probabilities
$\bigl(
  \begin{smallmatrix}
    p_{00} & p_{01} \\
    p_{10} & p_{11}
  \end{smallmatrix}
\bigr)$.
It is clear that we get independent random variables $(\tau_n)$
if raws of the matrix of transition probabilities are coincide.

We consider random variable
\[
  \tau = \tau_1 g_{1-\tau_1}
    + \sum_{k=2}^\infty \tau_k g_{1-\tau_k} \prod_{i=1}^{k-1} g_{\tau_i}
  \equiv \Delta^{G_2}_{\tau_1\tau_2...\tau_n...}
\]

It is certain that random variable $\tau$ is well defined.
The Lebesgue structure
(content of discrete, absolutely continuous and singular components)~\cite{Pratsiovytyi:1998:FPD:en},
spectral and fractal properties of distribution of $\tau$
are of interest to us.

It is evident that
\[
  P\{\xi = x_0 = \Delta^{G_2}_{\alpha_1\alpha_2...\alpha_n...}\}
  = p_{\alpha_1} p_{\alpha_1\alpha_2} \ldots p_{\alpha_n\alpha_{n+1}} \ldots,\]
if $x_0$ is a $G_2$-binary point, then
\begin{multline*}
  P\{\xi = \Delta^{G_2}_{\alpha_1...\alpha_{m-1}11(0)}
  = \Delta^{G_2}_{\alpha_1...\alpha_{m-1}01(0)}\} \\
  = P\{\xi = \Delta^{G_2}_{\alpha_1...\alpha_{m-1}11(0)}\}
    + P\{\xi = \Delta^{G_2}_{\alpha_1...\alpha_{m-1}01(0)}\}.
\end{multline*}

If there are two zeros among the entries of the matrix,
random variable $\tau$ has a discrete distribution with two atoms:

1) $\Delta^{G_2}_{(0)},\Delta^{G_2}_{(1)}$,
if
$p_{00}=1=p_{11}$;

2) $\Delta^{G_2}_{(01)}$, $\Delta^{G_2}_{(10)}$,
if
$p_{01}=1=p_{10}$;

3) $\Delta^{G_2}_{(0)}$, $\Delta^{G_2}_{1(0)}$,
if
$p_{00}=1=p_{10}$;

4) $\Delta^{G_2}_{0(1)}$, $\Delta^{G_2}_{(1)}$,
if
$p_{01}=1=p_{11}$.

\begin{lemma}
  Let $E$ be the set of numbers $x\in [0;g_0]$ such that pair of digits $(a,c)\in A^2$
  occurs only finitely many as consecutive digits in the $G_2$-representation.
  The set $E$ is of zero Lebesgue measure.
\end{lemma}

\begin{proof}
  1. Consider the set $C=C[G_2;\overline{ac}]$ of all numbers $x \in [0;g_0]$
  such that pair of digits $ac$ does not occurs in the $G_2$-representation.
  It is evident that
  $C\cap\nabla^{G_2}_{00}=\varnothing$.

  1) If $(a,c)=(0,1)$ or $(a,c)=(1,0)$,
  then the set $C$ is countable (see above).

  2) If $(a,c)=(0,0)$ or $(a,c)=(1,1)$,
  then $C$ is a self-similar set of Cantor type of zero Lebesgue measure.
  Now we will prove it.

  Without loss of generality, we make reasoning for the case $(a,c)=(0,0)$.
  It is evident that $C=C_1\cup C_{01}$, where
  $C_1\equiv \Delta^{G_2}_{1}\cap C$,
  $C_{01}\equiv \Delta^{G_2}_{01}\cap C$,
  and $C$ is similar to $C_1$ with ratio $|g_1|=-g_1$,
  $C$ is similar to $C_{01}$ with ratio $|g_0g_1|=-g_0g_1$.
  Then the following equalities hold for Lebesgue measure:
  \[
    \lambda(C) = \lambda(C_1) + \lambda(C_{01})
    = -g_1 \lambda(C) - g_0 g_1 \lambda(C)
    = (1 - g_0^2) \lambda(C).
  \]
  Hence $g_0^2\lambda(C)=0$, i.e., $\lambda(C)=0$.

  2. Let $E_n$ be the set of all numbers
  $x=\Delta^{G_2}_{\alpha_1...\alpha_n...}$
  such that
  $\overline{\alpha_m\alpha_{m+1}}\ne \overline{ac}$
  for $m>n$.
  Then $E_0=C[G_2;\overline{ac}]$ and $\lambda(E_0)=0$.
  Since
  \[
    E_n = \bigcup_{c_1\in A}...\bigcup_{c_n\in A} [\Delta^{G_2}_{c_1...c_m}\cap E_n]
  \]
  and we have proved already that
  $\lambda([\Delta^{G_2}_{c_1...c_m}]\cap E_n)=0$,
  we have
  $\lambda(E_n)=0$.

  Taking into account that $E=\bigcup\limits_{n=0}^{\infty}E_n$ and fact that
  union of sets of zero Lebesgue measure is a set of zero Lebesgue measure,
  we see that $\lambda(E)=0$.
\end{proof}

\begin{corollary}
  \label{cor1}
  Almost all (with respect to Lebesgue measure) numbers from interval $[0;g_0]$
  use every pair $00$, $01$, $10$, $11$ infinitely many as consecutive digits
  in the $G_2$-representation.
\end{corollary}

In fact, this statement says about the set $H\equiv [0;g_0]\setminus E$,
whose measure is equal to $g_0$.

\begin{remark}
  Corollary~\ref{cor1} expresses normal property of numbers from interval $[0;g_0]$
  in terms of $G_2$-representation.
\end{remark}

\begin{theorem}
  If matrix of transition probabilities $||p_{in}||$ contains only one zero,
  then random variable $\tau$ has

  1) a pure discrete distribution with a countable point spectrum
  \begin{equation}
    \label{eq2}
    D_\tau = \{\Delta^{G_2}_{(1-i)}, \;
      \Delta^{G_2}_{\underbrace{\scriptstyle ii...i}_k(1-i)}, \; k\in N\},
  \end{equation}
  if
  $p_{[1-i]i}=0$;

  2) singularly continuous distribution of Cantor type
  with a self-similar fractal spectrum
  \begin{align}
    \nonumber
    S_{\tau} &= C[G_2, \overline{[1-i][1-i]}] \\
    &= \{x=\Delta^{G_2}_{\alpha_1...\alpha_n...}, \;
      \alpha_k\alpha_{k+1}\ne \overline{[1-i][1-i]} \; \forall k\in N\},
  \end{align}
  whose Hausdorff--Besicovitch dimension is a solution of equation
  \begin{equation}
    \label{eq4}
    |g_i|^x+|g_0g_1|^x=1,
  \end{equation}
  if
  $p_{[1-i][1-i]}=0$.
\end{theorem}

\begin{proof}
  1) Let
  $p_{[1-i]i}=0$.
  Then
  $p_{[1-i][1-i]}=1$,
  and thus point
  $\Delta^{G_2}_{(1-i)}$
  is an atom with mass
  $p_{1-i}\cdot p_{[1-i][1-i]}=p_{1-i}$.

  For
  $p_{ii}\ne 0\ne p_{i[1-i]}$,
  we have
  $\Delta^{G_2}_{\underbrace{\scriptstyle ii...i}_{k}(1-i)}$
  is an atom with mass
  $p_i\cdot p_{ii}^kp_{i[1-i]}>0$,
  where
  $k\in N$.
  Moreover, sum of masses of all atoms is equal to
  \[
    p_{1-i} + p_ip_{i[1-i]} \sum_{k=0}^\infty p_{ii}^k
    = p_{1-i} + \frac{p_ip_{i[1-i]}}{1-p_{ii}}
    = p_{1-i}+\frac{p_ip_{i[1-i]}}{p_{i[1-i]}}=1.
  \]
  Hence, distribution of random variable $\tau$ is pure discrete.

  2) Since $P\{\tau=x_0\}=0$ for any $x_0\in [0;g_0]$,
  we see that distribution of $\tau$ is continuous.
  Let $p_{[1-i][1-i]}=0$.
  Then $p_{[1-i]i}=1$ and $P\{\tau=x_0\}=0$ for any $x_0\in [0;g_0]$,
  and thus $\tau$ has continuous distribution.
  Moreover,
  $P\{\tau\in \Delta^{G_2}_{i}\}=p_i>0$,
  $P\{\tau\in \Delta^{G_2}_{[1-i]i}\}=p_{1-i}\cdot p_{[1-i]i}>0$,
  $P\{\tau\in \Delta^{G_2}_{[1-i][1-i]}\}=p_{1-i}p_{[1-i][1-i]}=0$.

  Show that $S_\tau = C \equiv C[G_2, \overline{[1-i][1-i]}]$.
  Let $x\in C$, i.e.,
  $x=\Delta^{G_2}_{\alpha_1\alpha_2...\alpha_n...}$
  and
  $(\alpha_k,\alpha_{k+1})\ne (1-i,1-i)$
  for any $ k\in N$.

  Then
  \begin{gather*}
    P\{\tau\in \Delta^{G_2}_{\alpha_1\alpha_2...\alpha_m}\}
    = p_{\alpha_1}p_{\alpha_1\alpha_2}\cdot ...\cdot p_{\alpha_{m-1}\alpha_m}>0,\\
    P\{\tau\in \Delta^{G_2}_{c_1c_2...c_m}\}
    = p_{c_1}p_{c_1c_2}\cdot...\cdot p_{c_{m-1}c_m}=0,
  \end{gather*}
  if there exists
  $(c_k,c_{k+1})=(1-i,1-i)$.

  Hence, $S_{\tau}=C$ and $C$ is a self-similar set
  because of
  $C=\Delta'_i\cup \Delta'_{[1-i]i}$,
  where
  $\Delta^{G_2}_{i}\cap C=\Delta'_{i}\stackrel{k_1}{\sim}C$,
  $k_1=|g_i|$;
  $\Delta^{G_2}_{[1-i]i}\cap C\equiv \Delta'_{[1-i]i}\stackrel{k_2}{\sim}C$,
  $k_2=|g_0g_1|$.

  Equation for calculating self-similar dimension,
  which is equal to Hausdorff--Besicovitch dimension in this case,
  has a form
  $|g_i|^x+|g_0g_1|^x=1$,
  and it is obvious that its root is fractional.
\end{proof}

\begin{theorem}
  If there are no zeros among the entries of matrix of transition probabilities,
  then distribution of $\tau$ is pure continuous,
  and it is absolutely continuous only if $p_{ij}=|g_{j}|$.
  Moreover, if $p_0=g_0$,
  then it is a uniform distribution on interval $[0;g_0]$.
\end{theorem}

\begin{proof}
  1) If
  $p_0=g_0$, $p_1=-g_1$
  and
  $p_{00}=g_0=p_{10}$, $p_{01}=-g_1=p_{11}$,
  then, for any $m\in N$ and $(c_1,...,c_m)\in A^m$, we have
  \begin{align*}
    P\{\xi\in \Delta^{G_2}_{c_1...c_m}\}
    &= p_{c_1}p_{c_1c_2}\cdot ...\cdot p_{c_{m-1}c_m}
    = |g_{c_1}g_{c_2}...g_{c_m}| \\
    &= g_0^{-1}|\Delta^{G_2}_{c_1...c_m}|
    = g_0^{-1}\lambda(\Delta^{G_2}_{c_1...c_m}),
  \end{align*}
  i.e., distribution of random variable $\xi$ is geometric,
  namely, it is uniform on $[0;g_0]$.

  2) Suppose matrix of transition probabilities
  does not fulfill above-mentioned conditions.
  Then for at least one pair of digits $ij$ inequality $p_{ij}|q_j|^{-1}<1$ holds.
  Since distribution of $\tau$ does not have atoms,
  we see that function of distribution $F_{\tau}$ is continuous and monotonic.
  Thus, by a known Lebesgue theorem, it has a finite derivative
  on the set $W\subset [0;g_0]$ of full Lebesgue measure.
  By the previous lemma, the set $H$ of ``normal numbers''
  is also a set of full Lebesgue measure.
  Thus, their intersection $W\cap H$ is a such set too.

  Consider point $x_0\in W\cap H$.
  Since derivative $F'(x_0)$ exists and is finite,
  it can be calculated by formulae:
  \begin{align*}
    F'(x_0) &= \lim_{n\to\infty}
      \frac{P\{\xi\in \Delta^{G_2}_{\alpha_1(x_0)...\alpha_n(x_0)}\}}
           {|\Delta^{G_2}_{\alpha_1(x_0)...\alpha_n(x_0)}|} \\
    &= \lim_{n\to\infty} g_0
      \frac{p_{\alpha_1}p_{\alpha_1\alpha_2}...p_{\alpha_{n-1}\alpha_n}}
           {|g_{\alpha_1}g_{\alpha_2}...g_{\alpha_n}|}
    = g_0 \frac{p_{\alpha_1}}{g_{\alpha_1}}
      \prod_{i=1}^{\infty}\frac{p_{\alpha_i\alpha_{i+1}}}{q_{\alpha_{i+1}}}.
  \end{align*}
  Since $x_0\in H$, we see that condition
  $\frac{p_{\alpha_i\alpha_{i+1}}}{|q_{\alpha_{i+1}}|}\neq 1$
  holds infinitely many times.
  Hence, condition
  $\frac{p_{\alpha_i\alpha_{i+1}}}{|q_{\alpha_{i+1}}|}< 1$
  also holds infinitely many times.
  Thus, neccessary condition of convergence of infinite product is not fulfilled.
  Hence, $F'(x_0)=0$, i.e., distribution of $\tau$ is singularly continuous.
\end{proof}

\section{The Second Object.
  The Case of Independent Digits of~$G_2$-representation}

We consider random variable
$\xi=\Delta^{G_2}_{\xi_1\xi_2...\xi_n...}$,
where $(\xi_n)$ is a sequence of random variables
taking the values $0$ and $1$ with probabilities $p_{0n}$ and $p_{1n}$, respectively.
It is clear that structure~\cite{Pratsiovytyi:1998:FPD:en,Jessen:1935:DFR} and properties of distribution of $\xi$
is determined by infinite two-row stochastic matrix $||p_{ij}||$.

\begin{theorem}
  Random variable $\xi$ has either pure discrete or pure continuous distribution.
  Moreover, it is discrete if and only if
  \begin{equation}
    \label{eq5}
    M\equiv \prod_{n=1}^{\infty}\max\{p_{0n}, p_{1n}\}>0.
  \end{equation}
  In discrete case, point spectrum (set of atoms) of the random variable
  is a tail set whose representative is a point
  $x_0=\Delta^{G_2}_{c_1...c_m...}$
  such that
  $p_{c_nn}=\max\{p_{0n}, p_{1n}\}$,
  or its subset.
\end{theorem}

\begin{proof}
  If distribution of $\xi$ has atoms, then there exists point
  $x=\Delta^{G_2}_{\alpha_1\alpha_2...\alpha_n...}$
  such that
  $\prod\limits_{k=1}^{\infty}p_{\alpha_kk}>0$.
  Then
  \[
    0 < P\{\xi=x\} = \prod_{k=1}^{\infty} p_{\alpha_kk} \leq M.
  \]
  Thus, if $\xi$ has atoms, then condition~\eqref{eq5} holds.

  Now suppose that condition~\eqref{eq5} holds.
  Consider any point
  $x=\Delta^{G_2}_{a_1a_2...a_n...}$
  such that its $G_2$-representation differs from representation of number $x_0$
  at most finite number of first digits
  and $p_{a_kk}>0$ for any $k\in N$.
  It is evident that such $x$ is an atom of distribution of $\xi$.

  Consider the set $E_k$ of all numbers
  $x_i^{(k)}=\Delta^{G_2}_{a_1a_2...a_k...}$
  such that their $G_2$-representation coincides with representation of $x_0$
  starting from $(k+1)$th digit,
  and $p_{a_kk}>0$ for any $k\in N$.
  Then
  \begin{align*}
    P\{\xi\in E_k\} &= \sum_{a_1: p_{a_11}>0} ... \sum_{a_k: p_{a_kk}>0}
      (p_{a_11}p_{a_22}\cdot...\cdot p_{a_kk} \prod_{i=k+1}^{\infty}p_{c_ii}) \\
    &= (\prod_{i=k+1}^{\infty} p_{c_ii})
      \sum_{a_1: p_{a_11}>0}p_{a_11}\cdot...\cdot \sum_{a_k: p_{a_kk}>0}p_{a_kk}
    = \frac{M}{p_{c_11}p_{c_22}\cdot...\cdot p_{c_kk}}.
  \end{align*}
  It is evident
  $E_0\subset E_1\subset\cdots\subset E_n\subset E_{n+1}\subset\cdots$.
  The set
  $E=\bigcup\limits_{k=1}^{\infty}E_k=\lim\limits_{k\to \infty}E_k$
  is countable of finite
  (it can be even single-point if every column of the matrix contains $1$).
  Thus,
  \[
    P\{\xi\in E\} = \lim_{k\to\infty} P\{\xi\in E_k\}
    = \lim_{k\to\infty} \frac{M}{p_{c_11}p_{c_22}...p_{c_kk}} = 1.
  \]
  Hence, distribution of $\xi$ is concentrated on at most countable set,
  so it is discrete by definition.

  Since points beolonging to set $E_k$ and number $x_0$ have a common tail,
  we see that $E$ is a tail set with representative $x_0$ or its subset
  if matrix $||p_{ik}||$ contains zeros and, for some
  $x=\Delta^{G_2}_{a_1...a_kc_{k+1}...c_{k+n}...}$
  there exists
  $p_{a_ii}=0$.

  Since $M>0$ is a neccessary condition of discreteness of $\xi$,
  we see that distribution of $\xi$ does not have atoms, so, is continuous
  if $M=0$.
\end{proof}

\section{Normal Properties of Numbers in Terms of Frequencies of~Digits
  of Their $G_2$-representation}

Let
$x=\Delta^{G_2}_{\alpha_1...\alpha_k...}$.
Put
$N_1(x,k) \equiv \alpha_1 + ... + \alpha_k$, $N_0(x,k) \equiv k - N_1(x,k)$.

\begin{definition}
  A limit
  $\lim\limits_{k\rightarrow \infty} \frac{N_i(x,k)}{k} \equiv \nu_{i}(x)$,
  $i\in A$,
  if it exists (it is not always true),
  is called a frequency of digit $i$ in the $G_2$-representation of number $x$.
\end{definition}

Statistical base of the notion of frequency is evident
(at least conditions for its existence).
However, it is known that
this notion has a deep metric meaning for various representations of real numbers.

It is clear that if frequency of digit $i\in A$ exists,
then frequency of another digit $1-i$ exists,
and vice versa, if frequency of one digit does not exist,
then frequency of another digit does not exist.
Moreover, it is easy to prove (similarly to~\cite{Pratsiovytyi:1995:SSN:en})
that set of numbers that does not have frequencies of digits
in their $G_2$-representation
is a massive null set of cardinality of continuum.
\begin{theorem}[%
  An analog of the Borel theorem~\cite{Borel:1909:PDA} for binary representation%
  ]
  Let
  $B = \{x \colon \nu_0(x) = g_0\}$
  be a set of numbers $x$ from interval $[0;g_0]$ such that
  a frequency of digit $0$ in their $G_2$-representation is equal to $g_0$.
  Then the set $B$ is of full Lebesgue measure.
\end{theorem}

\begin{proof}
  Using methods of calculus, we prove
  that the set of numbers $x$ such that $\nu_0(x)\neq g_0$
  (i.e., frequency of $0$ does not exist or is not equal to $g_0$)
  is of zero Lebesgue measure.

  Condition $\nu_0(x)=g_0$ is equivalent to condition $\nu_1(x)=-g_1$.
  We write it in the form
  $\lim\limits_{k\rightarrow\infty}k^{-1}(\alpha_1(x)+\ldots+\alpha_k(x))=-g_1$,
  which is equivalent to
  \begin{equation}
    \label{teor_Bor_2}
    \lim_{k\rightarrow\infty} (k^{-1} \sum_{i=1}^k \alpha_i(x) + g_1) = 0.
  \end{equation}

  Consider integrals
  \begin{align*}
    I_k &= \int^{g_0}_0 (\frac{1}{k} \sum^{k}_{i=1} \alpha_i{(x)}+g_1)^2 dx \\
    &= \frac{1}{k^2} \int_0^{g_0} (\sum_{i=1}^k \alpha_i(x) + k g_1)^2 dx \\
    &= \frac{1}{k^2} \int^{g_0}_0 (\sum^k_{i=1} (\alpha_i(x)+g_1))^2 dx.
  \end{align*}

  There are two types of integrals in expression of $I_k$.
  The first one is
  \begin{align*}
    \int^{g_0}_0 (\alpha_i(x) + g_1)^2 dx
    &= \int^{g_0}_0 (\alpha^2_i(x) + 2 \alpha_i(x) g_1 + g_1^2) dx \\
    &= \int^{g_0}_0 (\alpha_i(x) (1+2g_1) + g_1^2) dx \\
    &= (1 + 2g_1) \int^{g_0}_0 \alpha_i(x) dx + \int^{g_0}_0 g_1^2 dx \\
    &= (1 + 2g_1) \lambda\{x \colon \alpha_i(x) = 1\} g_0 + g_0 g_1^2 \\
    &= (1 + 2g_1) g_0 (-g_1) g_0 + g_1^2 g_0 = g_0 g_1 (1 - 2 g_0 g_1),
  \end{align*}
  because of $\alpha_i^2(x)=\alpha_i(x)$
  (the number of such integrals is equal to $k$).

  The second one is
  \begin{align*}
    \int^{g_0}_0 (\alpha_{i}(x) &+ g_1) (\alpha_j(x) + g_1) dx \\
    &= \int^{g_0}_0 \alpha_i(x) \alpha_j(x) dx
      + g_1 \int^{g_0}_0 (\alpha_i(x) + \alpha_j(x)) dx
      + \int^{g_0}_0 g_1^2 dx \\
    &= g_0 g_1^2 + 2 g_1 g_0 (-g_1) + g_1^2 g_0 = 0
  \end{align*}
  for
  $i\neq j$.

  Hence,
  $I_k=\dfrac{g_0g_1(1-2g_0g_1)}{k}$.
  Then
  $\lim\limits_{k\to\infty} I_k=0$,
  i.e., sequence of sums
  $\dfrac{1}{k}\sum\limits^k_{i=1}\alpha_i(x)$
  converges in quadratic mean to $(-g_1)$.
  However, convergence in mean square
  does not imply almost everywhere convergence with respect to Lebesgue measure.

  For any sufficiently small $\varepsilon > 0$,
  we consider the set $E_k(\varepsilon)$ of numbers such that
  \begin{equation}
    \label{teo_Borel_1}
    |\dfrac{1}{k}\sum\limits^k_{i=1}\alpha_i(x)+g_1|>\varepsilon.
  \end{equation}

  In order to estimate the measure of this set, remark that
  \begin{align*}
    I_k &= \int^{g_0}_0 (\dfrac{1}{k} \sum^k_{i=1} \alpha_i(x) + g_1)^2 dx
    \geq \int_{E_k(\varepsilon)} (\dfrac{1}{k} \sum^k_{i=1} \alpha_i(x) + g_1)^2 dx \\
    &\geq \varepsilon^2 \int_{E_k(\varepsilon)} dx
    = \varepsilon \lambda[E_k(\varepsilon)].
  \end{align*}
  Hence
  \[
    \lambda[E_k(\varepsilon)]
    \leq \dfrac{I_k}{\varepsilon^2}
    = \dfrac{g_0g_1(1-2g_0g_1)}{k\varepsilon^2}.
  \]

  Thus, for any fixed $\varepsilon$, equality
  $\lim\limits_{k\rightarrow\infty}\lambda[E_k(\varepsilon)]=0$
  holds.
  However, it is not enough to make decision
  that (\ref{teor_Bor_2}) holds for almost all $x\in[0;g_0]$.

  Consider sequence of sets of numbers $x(\varepsilon)$:
  $E_1(\varepsilon)$,
  $E_4(\varepsilon)$,
  $E_9(\varepsilon)$,
  ...,
  $E_{n^2}(\varepsilon)$,
  ...
  and let $F_k(\varepsilon)$ be a set of numbers
  that belong to at least one of the sets
  $E_{k^2}(\varepsilon)$, $E_{(k+1)^2}(\varepsilon)$, ...
  Taking into account relation~(\ref{teo_Borel_1}), we have
  \begin{align*}
    \lambda[F_k(\varepsilon)]
    &= \lambda [E_{k^2}(\varepsilon)\cup E_{(k+1)^2}(\varepsilon)\cup...] \\
    &\leq \sum^\infty_{j=0} \lambda[E_{(k+j)^2}(\varepsilon)]
    \leq \dfrac{g_0g_1(1-2g_0g_1)}{k\varepsilon^2}
    \sum^\infty_{j=0} \dfrac{1}{(k+j)^2} \rightarrow 0
    \quad (k\rightarrow\infty).
  \end{align*}

  Since sequence of sets $\{F_k(\varepsilon)\}$ is monotonic,
  i.e., $F_{k+1}(\varepsilon)\subset F_k(\varepsilon)$,
  and Lebesgue measure of the set $F_k(\varepsilon)$ tends to zero as $k \to \infty$,
  we see that measure of the intersection of all these sets is equal to $0$.
  It is equivalent to the fact that all numbers except for numbers from a null set
  can belong to finite collection of these sets only.
  It means if point $x$ belongs to finite number of sets $F_k(\varepsilon)$ only,
  then, for sufficiently large $k$, it does not belong to set $F_k(\varepsilon)$,
  thus, it does not belong to any of the sets $E_{n^2}(\varepsilon)$.
  So, for any number $k$ that is greater than some $k_0$, we have
  \[
    |\dfrac{1}{k^2} \sum^{k^2}_{i=1} \alpha_i(x) + g_1| \leq \varepsilon.
  \]

  Thus, this property holds for almost all numbers
  for any fixed $\varepsilon$.
  Hence, for almost all $x$, we have
  \[
    \dfrac{1}{k^2} \sum^{k^2}_{i=1} \alpha_i(x) \to -g_1
    \quad
    \text{as}
    \quad
    k \to \infty.
  \]

  We have proved the statement for the case
  if numbers $k$ increase as a sequence of squares of positive intergers.
  If $k$ is an arbitrary positive integer number,
  then there exists number $m\in N$ such that
  \[
    m^2\leq k<(m+1)^2,
    \quad
    \text{i.e.,}
    \quad
    0\leq k-m^2< 2m+1.
  \]
  Hence,
  \[
    \frac{1}{k} \sum^k_{i=1} \alpha_i(x)
    = \frac{1}{k} (\sum^{m^2}_{i=1} \alpha_i(x) + \sum^{k}_{i=m^2+1} \alpha_i(x))
    = \frac{1}{m^2} \sum_{i=1}^{m^2} \alpha_i(x) \frac{m^2}{k}
      + \frac{1}{k} \sum^{k}_{i=m^2+1} \alpha_i(x).
  \]

  As we have proved already, as $k\rightarrow\infty$,
  \[
    \dfrac{1}{m^2} \sum^{m^2}_{i=1} \alpha_i(x) \rightarrow -g_1
  \]
  almost everywhere and
  \[
    \dfrac{m^2}{(m+1)^2}<\dfrac{m^2}{k}\leq1
  \]
  everywhere. Hence,
  \[
    \dfrac{1}{k} \sum^{k}_{i=m^2+1} \alpha_i(x)
    \leq \dfrac{k-m^2}{k}
    < \dfrac{2m+1}{m^2}.
  \]
  As a result,
  $\dfrac{m^2}{k}\rightarrow 1$
  and
  $\dfrac{1}{k} \sum\limits^{k}_{i=m^2+1} \alpha_i(x) \rightarrow 0$
  as
  $k\rightarrow\infty$.
  Thus equality (\ref{teor_Bor_2}) holds almost everywhere.
\end{proof}

\begin{corollary}
  The set $B[G_2;\nu_0]\equiv\{x: \nu_0(x)=q_0\ne g_0\}$ is of zero Lebesgue measure
  if $\nu_0\ne g_0$.
\end{corollary}

\begin{remark}
  Property of number $x\in[0;g_0]$ to have a frequency of $0$ in the $G_2$-representation
  that is equal to $g_0$ (equivalently, $\nu_1(x)=g_1)$
  is normal.
\end{remark}

\section{Continuous Random Variable
  with Independent Digits of~$G_2$-representation}

\begin{theorem}
  In continuous case ($M=0$), random variable
  $\xi=\Delta^{G_2}_{\xi_1\xi_2...\xi_n...}$
  with independent digits of its representation has either pure absolutely continuous
  or pure singular distribution (with respect ot Lebesgue measure).
  Moreover, it is absolutely continuous if and only if
  \[
    S = \sum_{n=1}^\infty [(1-\frac{p_{0n}}{g_0})^2 + (1+\frac{p_{1n}}{g_1})^2]
    < \infty.
  \]
\end{theorem}

\begin{proof}
  Let $T_i$ be an right-shift operator for digits
  of $G_2$-representation of numbers with parameter $i$, i.e.,
  \[
    T_i(x=\Delta^{G_2}_{\alpha_1\alpha_2...\alpha_n...})
    \equiv \Delta^{G_2}_{i\alpha_1\alpha_2...\alpha_n...}.
  \]
  $T_i(E)$-transformation of a set $E$ is the set
  $T_i(E)\equiv\{x'=T_i(x), x\in E\}$.

  Let $(i_1,i_2,...,i_m)\equiv \theta$ be a fixed tuple of zeros and ones.
  Put
  \[
    T^m_{\theta}(E)=\{T_{\theta}^{m}(x)=T_{i_1}(T^{m-1}_{i_2...i_{m}}(x)), x\in E\}
  \]
  and $T_{\theta}^{m}$ is called a transformation of a set $E$.

  It is easy to see that
  $T_{\theta}^{m}([0;g_0])=\Delta^{G_2}_{i_1...i_m}$
  and $T_{\theta}^{m}$ is a similarity transformation with ratio
  $k=\prod\limits_{j=1}^{m}|g_{i_j}|$.
  Then for Lebesgue measure we have
  $\lambda[T_{\theta}^{m}(E)]=\lambda(E)\prod\limits_{j=1}^{m}|g_{i_j}|$
  and values $\lambda[T_{\theta}^m(E)]$ and $\lambda(E)$ are equal to zero
  simultaneously.

  By $T^{m}(x)$ denote the set of all images of $x$ under transformation
  $T_{\theta}^{m}(x)$,
  where $\theta$ runs the whole set $A^m$ of ordered $m$-tuple
  of zeros and ones.

  Let $E$ be some Borel set from $[0;g_0]$,
  $T^{0}(x)\equiv x$,
  and $T$ the set of all possible transformations $T^m$
  for all finite values of $m$.

  Consider event $\Lambda=\{\xi\in T(E)\}$.
  It is evident that $\Lambda$ is a residual event
  with respect ot all $\sigma$-algebras $\mathbb{B}_k$
  generated by first $\xi_1$, ..., $\xi_k$ (for any finite $k$).
  Then, by $0$--$1$ law, probability $P(\Lambda)=0$ or $P(\Lambda)=1$.

  If there exists number $a$ such that $P\{\xi=a\}>0$ exists,
  then, considering set $\{a\}$ as $E$,
  we have $P(\Lambda)=1$ and $\lambda(\Lambda)=0$,
  i.e., probability $P$ is concentrated on at most countable set $\Lambda$
  and probability distribution is pure discrete.

  If such a number $a$ does not exist, then we consider the following cases:

  1) there exists set $E$ of zero Lebesgue measure such that $P\{\xi\in E\}>0$;

  2) such a set $E$ does not exist,
  i.e., for every $E$, from $\lambda(E)=0$ it follows that $P\{\xi\in E\}=0$.

  In the first case $P\{\xi \in T(E)\}=1$ and $\lambda\{T(E)\}=0$,
  i.e., $\xi$ has a pure singular distribution;
  in the second case it has a pure absolutely continuous distribution,
  by definition.

  Since
  $\varphi(\Delta^{G_2}_{\alpha_1\alpha_2...\alpha_n...})
  =\Delta^{G_2}_{0\alpha_1\alpha_2...}$
  is a measurable mapping preserving Lebesgue measure,
  \[
    \Delta^{G_2}_{a_1a_2...a_n...}
    \equiv a_{\alpha_1}q_{1-\alpha_{a_1}}
      + \sum_{k=2}^\infty a_{\alpha_k} q_{1-\alpha_k} \prod_{i=1}^{k-1}q_{\alpha_i},
  \]
  where $q_0=g_0$, $q_1=|g_1|$,
  we see that criterion of absolutely continuity (as well as singularity)
  is a consequence of the known criterion for $Q_2$-representation of number~\cite{Pratsiovytyi:1998:FPD:en}.
\end{proof}

\section{The Case of Identically Distributed Digits}

Consider a generalization of $\alpha$-dimensional Hausdorff measure
introduced by Billingsley.
Let $\mu$ be a continuous probability measure
on $\sigma$-algebra of Borel subsets of interval.
The value of expression
\[
  H_{\mu}^{\alpha}(E)
  = \lim_{\varepsilon\to 0} \inf_{\mu(u_j)\leq\varepsilon}
    \{\sum_{j}\mu^{\alpha}(u_j)\},
\]
where infimum is taking by all at most countable
$\mu$-$\varepsilon$-coverings $\{u_j\}$ of a set $E$ by intervals $u_j$,
i.e., $\mu(u_i)\leq\varepsilon$,
is called a \emph{$H_{\mu}^{\alpha}$-Hausdorff--Billingsley measure} of the set $E$.

Non-negative number
\[
  \alpha_{\mu}(E) = \sup\{\alpha:~H_{\mu}^{\alpha}(E)=+\infty\}
  = \inf\{\alpha:~H_{\mu}^{\alpha}(E)=0\}
\]
is called a \emph{Hausdorff--Billingsley dimension} of the set $E$
with respect to measure $\mu$.

It is easy to prove that in order to determine Hausdorff--Besicovitch dimension
of Borel sets from interval $[0;g_0]$
it is enough to use coverings by $G_2$-cylinders.
Similarly, in order to determine Hausdorff--Billingsley dimension,
when $\mu$ is the probability measure
corresponding to continuous distribution of random variable
with independent identically distributed digits of its $G_2$-representation
it is also enough to use coverings by $G_2$-cylinders.

\begin{theorem}
  Hausdorff--Besicovitch dimension of the set
  $B[G_2; p_0]=\{x: \nu_0(x)=p_0\}$
  is calculated by formula
  \begin{equation}
    \label{eq8}
    \alpha_0(E)=\frac{\ln{p_0^{p_0}p_1^{p_1}}}{\ln g_0^{p_0}(-g_1)^{p_1}}.
  \end{equation}
\end{theorem}

\begin{proof}
  Let $\Delta_n(x)$ be a $G_2$-cylinder of rank $n$ containing number $x$.
  We will use the following theorem
  generalizing the Billingsley theorem~\cite{Billingsley:1965:ETI}
  on determination of Hausdorff--Besicovitch dimension
  using their coverings by binary cylinders:
  ``Let $\mu$ and $\nu$ be continuous probability measures on $[0;g_0]$
  such that all $G_2$-cylinders are enough for determination
  of Hausdorff--Billingsley dimension of sets from this interval
  Then if
  \[
    E\subset \{x: \lim_{n\to\infty}\frac{\ln\nu(\Delta_n(x))}{\ln\mu(\Delta_n(x))}\}
    = \delta,
  \]
  we have $\alpha_{\mu}(E)=\delta\cdot\alpha_{\nu}(E)$.''
  To this end we take $\mu=g_0^{-1}\lambda$, where $\lambda$ is Lebesgue measure,
  and $\nu$ be a probability measure of Borel sets from interval $[0;g_0]$
  defined on $G_2$-cylinders by equality
  $\nu(\Delta^{G_2}_{\alpha_1\alpha_2...\alpha_n})=\prod\limits_{i=1}^{n}p_{\alpha_i}$.
  Since
  \[
    \nu(\Delta^{G_2}_{\alpha_1\alpha_2...\alpha_n})=p_0^{N_0(x,n)}p_1^{N_1(x,n)},
    \mbox{ where }
    N_1(x,n)=\alpha_1(x)+...+\alpha_n(x);
  \]
  \[
    \mu(\Delta^{G_2}_{\alpha_1\alpha_2...\alpha_n})=g_0^{1+N_0(x,n)}|g_1|^{N_1(x,n)},
    \mbox{ where }
    N_1(x,n)=n-N_0(x,n),
  \]
  we have
  \begin{align*}
    \lim_{n\to\infty}\frac{\ln\nu(\Delta^{G_2}_{\alpha_1\alpha_2...\alpha_n})}
    {\ln\mu(\Delta^{G_2}_{\alpha_1\alpha_2...\alpha_n})}
    &= \lim_{n\to\infty}\frac{\ln p_0^{N_0(x,n)}p_1^{N_1(x,n)}}
      {\ln g_0^{1+N_0(x,n)}|g_1|^{N_1(x,n)}} \\
    &= \lim_{n\to\infty}\frac{\ln\left(p_0^{\frac{N_0(x,n)}{n}}\cdot p_1^{\frac{N_1(x,n)}{n}}\right)^n}
      {\ln\left(g_0^{\frac{1+N_0(x,n)}{n}}\cdot |g_1|^{\frac{N_1(x,n)}{n}}\right)^n}
    = \frac{\ln p_0^{p_0}p_1^{p_1}}{\ln g_0^{p_0}|g_1|^{p_1}}.
  \end{align*}
  Since the set $B[G_2;p_0]$ is a set of full measure $\nu$,
  we have~\eqref{eq8}.
\end{proof}

\begin{theorem}
  If digits $(\xi_n)$ of continuous random variable
  $\hat{\xi}=\Delta^{G_2}_{\xi_1\xi_2...\xi_n...}$
  are independent and identically distributed,
  i.e., $p_{ik}=p_i$, $i=0,1$,
  then distribution of $\hat{\xi}$ is singularly continuous,
  and essential support of distribution
  is a fractal set of Besicovitch--Eggleston type $B[G_2;p_0]$
  of Hausdorff--Besicovitch dimension
  $\alpha_0(B)=\frac{\ln{p_0^{p_0}p_1^{p_1}}}{\ln{g_0^{p_0}(-g_1)^{p_1}}}$.
\end{theorem}

\begin{proof}
  The set $B[G_2;p_0]$ is a set of full probability measure $\nu$
  corresponding to random variable $\hat{\xi}$.
  However, for $p_0\neq g_0$, the set $B[G_2;p_0]$ is of zero Lebesgue measure.
  Hence, distribution of $\hat{\xi}$ is singular.

  Since, for $p_0\neq g_0$, Hausdorff--Billingsley of the set $B[G_2;p_0]$
  fulfill inequality
  \[
    0 < \alpha_0(B)
    = \frac{\ln{p_0^{p_0}(1-p_0)^{1-p_0}}}{\ln{g_0^{p_0}(-g_1)^{1-p_0}}}
    < 1,
  \]
  we see that distribution of random variable $\hat{\xi}$
  is a fractal singular distribution of Salem type.
\end{proof}

\begin{corollary}
  Different distribution of random variables $\hat{\xi}$ and $\hat{\xi}'$
  such that their $G_2$-digits are independent and identically distributed
  are mutually orthogonal.
\end{corollary}


\end{document}